\documentclass[12pt,reqno]{amsart}
\usepackage{amsmath,amssymb,url,fullpage,hyperref}
\usepackage[all]{xy}

\theoremstyle{definition}
\newtheorem{thm}{Theorem}

\numberwithin{thm}{section}
\newtheorem{defn}[thm]{Definition}     
\newtheorem{prop}[thm]{Proposition}
\newtheorem{lemma}[thm]{Lemma}
\newtheorem{construction}[thm]{Construction}
\newtheorem{rem}[thm]{Remark}
\newtheorem{notation}[thm]{Notation}
\newtheorem{ex}[thm]{Example}

\newcommand{\gpoid}[3]{\xymatrix{#1=\{#2 \ar@<3pt>[r] \ar@<-3pt>[r] & #3\}}}

\newcommand{\cchain}{{\sf 2\mydash term}}
\newcommand{\VB}{{\sf VB}}

\newcommand{\id}{\mathrm{id}}
\newcommand{\Hom}{\mathrm{Hom}}

\newcommand{\act}{{\sf Act}}

\newcommand{\R}{\mathbb{R}}

\newcommand{\WR}{{\sf WRep}}
\newcommand{\ruth}{{\sf RutH_2}}

\usepackage[normalem]{ulem}
\newcommand{\mydash}{\hbox{\sout{ }}} 

\title{Weak representations, representations up to homotopy, and VB-groupoids}
\author{Seth Wolbert}
\address{Department of Mathematics, University of Illinois, Urbana,
  IL 61801}
  \email{wolbert2@illinois.edu}

\begin{document}
  \begin{abstract}
  In this paper, I introduce weak representations of a Lie groupoid $G$.  I also show that there is an equivalence of categories between the categories of 2-term representations up to homotopy and weak representations of $G$.  Furthermore, I show that any VB-groupoid is isomorphic to an action groupoid associated to a weak representation on its kernel groupoid; this relationship defines an equivalence of categories between the categories of weak representations of $G$ and the category of VB-groupoids over $G$.
\end{abstract}
\maketitle
\section{Introduction}
For the duration of this paper, fix a Lie groupoid $\gpoid{G}{G_1}{G_0}$.  Representations up to homotopy of a Lie groupoid were first defined by Arias Abad and Crainic in \cite{AC}.  Their definition, designed as a global model of their representations of Lie algebroids, \cite{ACAlg}, allow one to address the inability to generalize many constructions available for group actions.  Indeed, there is no natural Lie groupoid action of a groupoid $G$ on its algebroid $A$ which one may call the adjoint action.  However, as noted in \cite{AC}, there is a natural notion of an adjoint representation up to homotopy.

Roughly speaking, one may think of representations up to homotopy as groupoid actions that fail to satisfy the axioms of an action in a controlled manner.  For a Lie groupoid $G$, a representation up to homotopy is formally defined as degree one operator $D$ on the bigraded right $C^\bullet(G)$-module $C(G;E)^\bullet$ associated to a graded vector bundle $E^\bullet \to G_0$.  This degree one operator must satisfy $D^2=0$ as well as a certain Leibnitz type rule (see Section \ref{s:Background} for more details).  

In practice, one generally works with the homogeneous pieces of this degree one operator.  These may be interpreted as a cochain structure $\delta^i:E^i\to E^{i+1}$ between the pieces of $E^\bullet$, a smooth collection of chain maps $\lambda_g^\bullet:E^\bullet \to E^\bullet$ (called a {\em quasi-action} of $G$), a smooth collection of chain homotopies $\Omega_{g,h}:\lambda^\bullet_g\circ \lambda^\bullet_h \Rightarrow \lambda^\bullet_{gh}$, and higher order coherence pieces.

Gracia-Saz and Mehta \cite{GM} showed that to any 2-term representation up to homotopy corresponds to a {\em VB-groupoid over $G$}: a map of Lie groupoids 
\[\pi:\gpoid{V}{V_1}{V_0} \to \gpoid{G}{G_1}{G_0}\]
(i.e., a smooth functor) such that the corresponding maps $\pi_1:V_1\to G_1$ and $\pi_0:V_0\to G_0$ on morphisms and objects are vector bundles and the structure maps are all maps of vector bundles covering the structure maps of $G$.  VB-groupoids were first introduced by Pradines \cite{P}.  A simple example is the tangent groupoid associated to any Lie groupoid.    

In addition to building a VB-groupoid from any representation up to homotopy, Gracia-Saz and Mehta showed that this establishes a bijection, up to isomorphism between 2-term representations up to homotopy of $G$ and VB-groupoids over $G$.  This was result was later extended by del Hoyo and Ortiz \cite{dHO}, who showed that this equivalence in fact corresponds to an equivalence of categories between the categories of 2-term representations up to homotopy of $G$ and VB-groupoids over $G$.

There are some technical issues with representations up to homotopy.  Indeed, while the homogeneous pieces may be interpreted as quasi-actions and smooth collections of homotopy operators, there are higher coherence conditions on these objects that defy geometric interpretation.  In addition, there are some notable issues with defining right representations up to homotopy.  

Also at issue is the lack of greater context.  Indeed, the definition of a representation up to homotopy relies heavily on the linear structure of the associated graded vector bundle.  It is not clear in general what an ``action up to homotopy'' on a more general graded manifold should look like.

On the other hand, weak representations of $G$ do not share these issues.  For a weak representation, the ``space'' $G$ acts on is replaced with a {\em linear Lie groupoid bundle} $\pi:V\to G_0$; roughly speaking, one should think of $V$ as a smooth collection of linear Lie groupoids.  $G$ acts on $V$ via a {\em weak action}: this is an action where each groupoid morphism $g\in G_1$ acts via linear equivalences of categories $A_g:\pi^{-1}(s(g))\to \pi^{-1}(t(g))$ while each pair $(g,h)$ of composable arrows corresponds to a linear natural isomorphism between $A_g\circ A_h$ and $A_{gh}$.  One may equivalently think of this as a (smooth) pseudofunctor from $G$ to the 2-category of equivalences between the fibers of $V$.

The main purpose of this paper is to show that the category of 2-term representations up to homotopy of $G$ $\ruth(G)$ and the category of weak representations of $G$ $\WR(G)$ are equivalent.  Additionally, we show that the category of weak representations of $G$ $\WR(G)$ and the category of $\VB(G)$ are equivalent.  This equivalence is naturally induced by the construction of action groupoids associated to each weak action of $G$ on a groupoid. 

This paper is organized as follows: in Section \ref{s:Background}, we discuss the background associated to representations up to homotopy and VB-groupoids in greater detail.  After this has been done, we may give a more detailed description as to the potential issues with the definition of representation up to homotopy.  In Section \ref{s:2-termandLLGB}, we discuss the equivalence between 2-term complexes of vector bundles and linear Lie groupoid bundles.  In Section \ref{s:WReps}, we define weak groupoid actions and weak representations.  Additionally, we show that the equivalence between 2-term complexes of vector bundles and Lie groupoid bundles extends to an equivalence between the categories of 2-term representations up to homotopy of $G$ and weak representations of $G$.  Finally, in Section \ref{s:VBsandWR}, we define the action groupoids associated to weak groupoid actions and show this construction defines an equivalence of categories between the category of weak representations up to homotopy of $G$ and VB-groupoids over $G$. 
\vskip 5 pt
\noindent {\bf Notation and Conventions:}  As stated above, for the duration of this paper, fix a Lie groupoid $\gpoid{G}{G_1}{G_0}$.  Let $s,t:G_1\to G_0$ be the source and target maps of $G$.  We assume that the reader is somewhat familiar with the definition of 2-categories.  2-categories and their functors are always assumed to be strict.  For all $n>0$, denote by $G_n$ the collection of $n$ composable arrows; i.e., $G_p:=\{(g_1,\ldots, g_n) \in G_1\times\ldots G_1 \; | \; s(g_i)=t(g_{i+1})\}$.
\vskip 5 pt
\noindent{\bf Acknowledgments:}  The author would like to acknowledge Eugene Lerman, Rui Loja Fernandes, and Daniel Berwick-Evans for numerous insightful conversations regarding representations up to homotopy and 2-category theory. 

\section{Background}\label{s:Background}

In this section, I will briefly discuss representations up to homotopy and their relationship to VB-groupoids.  Essentially, this is a collection of results from \cite{AC}, \cite{dHO}, and \cite{GM} that will motivate the remainder of the paper.

\begin{defn}
For each $p$, define the maps $s_p,t_p:G_p\to G_0$ by $t_p(g_1,\ldots, g_p)=t(g_1)$ and $s_p(g_1,\ldots, g_p)=s(g_p)$.  In other words, these are the source and targets of the arrow resulting from the product of the string of composable arrows $(g_1,\ldots, g_p)$.  

For every $n\geq 0$, {\em the space of degree $n$ groupoid cochains} is the space of smooth maps $C^n(G):=C^\infty(G_n)$ with the standard $\R$-vector space structure.  The coboundary map $\delta:C^\bullet(G)\to C^{\bullet+1}(G)$ for this cochain complex is defined as follows:
\begin{itemize}
\item For $n=0$: $\delta f(g):=f(s(g))-f(t(g))$.
\item For $n>0$: 
\begin{align*}\displaystyle \delta f (g_0,\ldots, g_{n})&:=\\ f(g_1,\ldots, g_n) + \sum_{i=1}^n (-1)^i f(g_0,\ldots, g_{i-1}g_i, \ldots g_n) &+(-1)^{n+1}f(g_0,\ldots, g_{n-1}).\end{align*}
\end{itemize}  
\end{defn}

\begin{defn}
Let $\pi:E\to G_0$.  Then {\em the space of degree $n$ groupoid cochains valued in $E$} is the vector space of sections $C^n(G;E):=\Gamma(t_n^*E)$ (or, in the case where $n=0$, simply the collection of sections $\Gamma(E)$). 

Note that $C^n(G;E)$ is a right $C(G)$-module: for $\omega\in C^p(G;E)$ and $f\in C^q(G)$, $\omega \star f$ is the degree $p+q$ groupoid cochain valued in $E$ defined by
\[(\omega \star f)(g_1,\ldots, g_{p+q}):= \omega(g_1,\ldots, g_p)\star f(g_{p+1}, \ldots, g_{p+q})\]

$C^\bullet(G;E)$ contains as a subspace {\em the normalized groupoid cochains valued in $E$}.  These are cochains $\eta\in C^n(G;E)$ for which $\eta(g_1,\ldots, g_n)=0$ whenever one of the $g_i$ is a unit.  More succinctly, for each degeneracy map $s_i:G_{n-1}\to G_n$, we require that $s_i^*\eta=0$.  The space of all degree $n$ normalized groupoid cochains is denoted by $\widehat{C}^n(G;E)$.  
\end{defn}

\begin{defn}
Let $\pi:E\to G_0$ be a vector bundle.  Then a {\em quasi-action} of $G$ on $\pi:E\to G_0$ is a smooth map $\lambda:G_1\times_{s,G_0,\pi}G_1\to E$, such that the restriction $\lambda_g:=\lambda(g,\cdot)$ for each $g\in G$ is a linear map $\lambda_g:E_{s(g)}\to E_{t(g)}$.  

A quasi-action $\lambda$ is {\em unital} if for every $x\in G_0$ with associated unit $u(x)$, $\lambda_{u(x)}=\mathrm{id}_{E_x}$.
\end{defn}

\begin{prop}\label{p:acttoop}
Suppose that $\pi:E\to G_0$ is a vector bundle.  Then there is a bijective correspondence $\lambda\mapsto D_\lambda$ between quasi-actions $\lambda$ of $G$ on $E$ and linear degree one operators $D_\lambda$ on $C^\bullet(G;E)$ satisfying the following graded Leibniz identity:
\[D_\lambda(f\star\omega)=  \omega \star \delta(f)+ (-1)^{|\eta|} D_\lambda\omega \star f.\]
Furthermore:
\begin{itemize}
\item The quasi-action $\lambda$ is unital if and only if the degree one operator $D_\lambda$ preserves $\widehat{C}^\bullet(G;E)$
\item The quasi-action $\lambda$ is an action if and only if it is unital and $(D_\lambda)^2=0$.
\end{itemize} 
\end{prop}

This proposition motivates the definition of a representation up to homotopy.  

\begin{defn}
Suppose that $\displaystyle \pi:E=\bigoplus_{i=0}^nE^i \to G_0$ is a graded vector bundle.  For each $n$, define $\displaystyle C(G;E)^n:=\bigoplus_{i=0}^n C^{i}(G;E^{n-i})$.  Then {\em a representation up to homotopy} is a {\em total} linear degree one operator $D$ on $C(G;E)^\bullet$ such that
\begin{itemize}
\item $D^2=0$; and
\item $D$ satisfies the graded Leibnitz identity:
\[D(f\star\omega)=  \omega \star \delta(f)+ (-1)^{|\eta|} D\omega \star f.\]
Here, $|\eta|$ refers to the {\em total} degree of $\eta$; i.e., for $\eta\in C^k(G;E^l)$, $|\eta|=k+l$.
\end{itemize} 

A {\em map of representations up to homotopy} between representations up to homotopy $(C(G;E)^\bullet, D)$ and $(C(G;E')^\bullet, D')$ is a linear map of right $C^\bullet(G)$-modules $\phi:C(G;E)^\bullet \to C(G;E')^\bullet$ such that $\phi\circ D=D'\circ \phi$.  
\end{defn}

By restricting the above degree one operator to homogeneous elements of $C(G;E)^\bullet$, one may decompose $D$ into a sequence of operators.  As we are only interested in the case of 2-term representations up to homotopy for this paper, we will only bother making this explicit in this case.  To do this, we must define transformation 2-cochains.

\begin{defn}
Let $\pi:E\to G_0$ and $\varpi:F\to G_0$ be two vector bundles over $G_0$.  Then a {\em transformation cochain from $E$ to $F$} is a section $\Omega$ of the vector bundle $\mathrm{Hom}(s_2^*E,t_2^*F)\to G_2$.  In other words, $\Omega$ associates to each pair of composable arrows $(g,h)\in G_2$ a linear map $\Omega_{g,h}:=\Omega(g,h):E_{s(h)}\to F_{t(g)}$.

A transformation cochain $\Omega$ from $E$ to $F$ is {\em normalized} if $\Omega(g,h)=0$ whenever $g$ or $h$ is a unit.    
\end{defn}

Now, we may describe 2-term representations up to homotopy in terms of homogeneous pieces:

\begin{prop}\label{propruth}
A {\em representation up to homotopy} of $G$ on a graded vector bundle $E^0\oplus E^1\to G_0$ is a uniquely determined by the tuple $(\delta:E^0\to E^1,\lambda^0,\lambda^1, \Omega)$ for $\delta:E^0\to E^1$ a map of vector bundles, $\lambda^0$ and $\lambda^1$ unital quasi-actions of $G$ on $E^0$ and $E^1$, and a normalized transformation cochain $\Omega$ from $E^1$ to $E^0$ satisfying for every triple of composable arrows $(g_1,g_2,g_3)\in G_3$:
\begin{itemize}
\item $\delta\circ \lambda_{g_1}^0 = \lambda_{g_1}^1\circ \delta$ 
\item $\lambda^0_{g_1g_2}-\lambda^0_{g_1} \circ\lambda^0_{g_2} = \Omega_{g_1,g_2}\circ\delta$
\item $\lambda^1_{g_1g_2}-\lambda^1_{g_1} \circ\lambda^1_{g_2} = \delta\circ\Omega_{g_1,g_2}$
\item $\lambda^0_{g_1}\circ\Omega_{g_2,g_3}-\Omega_{g_1g_2,g_3}+\Omega_{g_1,g_2g_3}-\Omega_{g_1,g_2}\circ\lambda^1_{g_3}=0$
\end{itemize}
\end{prop}

\begin{proof}
See Theorem 2.1, \cite{GM}.
\end{proof}

\begin{rem}\label{r:ruth}
The first condition of Proposition \ref{propruth} says $\delta:E^0\to E^1$ is equivariant with respect to the quasi-actions on $E^0$ and $E^1$.  Alternatively, it says that, thinking of $\delta:E^0\to E^1$ as a length 2 cochain complex of vector bundles over $G_0$, $\lambda_g^0$ and $\lambda_g^1$ define an automorphism of cochain complexes $\lambda_g$ for every $g\in G_1$.  

The second and third conditions of Proposition \ref{propruth} say that $\Omega_{g_1,g_2}$ functions as a chain homotopy relating the automorphisms of cochain complexes $\lambda_{g_1g_2}$ and $\lambda_{g_1}\lambda_{g_2}$ for every pair of composable arrows $(g_1,g_2)$.  

Finally, the fourth condition of Proposition \ref{propruth} is a cocycle type condition showing up as an extra consequence from the $D^2=0$ condition.  This has an interesting interpretation in terms of weak actions of groupoids.  
\end{rem}

One may also decompose maps of representations up to homotopy in an effective manner.

\begin{prop}\label{propmapofruth}
Suppose that $(C(G;E)^\bullet,D)$ and $C(G;E')^\bullet,D')$ are representations up to homotopy for graded vector bundles $E=E^0\oplus E^1\to G_0$ and $E'=\oplus E'^0\oplus E'^1\to G_0$.  If $(\delta:E^0\to E^1,\lambda^0,\lambda^1, \Omega)$ and $(\delta':E'^0\to E'^1, \lambda'^0,\lambda'^1, \Omega')$ are the unique tuples associated to $(C(G;E)^\bullet,D)$ and $C(G;E')^\bullet,D')$ (see Proposition \ref{propruth}), then maps of representations up to homotopy $\phi:(C(G;E)^\bullet,D) \to C(G;E')^\bullet,D')$ are in bijective correspondence with triples $(\phi^0:E^0\to E'^0,\phi^1:E^1\to E'^1,\mu)$ for $\phi^i$ maps of vector bundles and $\mu\in C^1(G;s^*\Hom(E^1,E'^0))$ satisfying for all composable $(g,h)\in G_2$:
\begin{itemize}
\item $\phi^1\circ \delta = \delta'\circ \phi^0$
\item $\phi^0\circ \lambda^0_g-\lambda'^0_g\circ \phi^0 = \mu_g\circ \delta$
\item $\phi^1\circ \lambda^1_g-\lambda'^1_g \circ \phi_0 = \delta'\circ \mu_g$
\item $\phi^0\circ \Omega_{g,h}+ \mu_g\circ \lambda^1_h + \lambda'^0_g\circ \mu_h = \mu_{gh} + \Omega_{g,h}' \circ \phi^1$
\end{itemize} 
\end{prop} 

\begin{proof}
See Proposition 2.6, \cite{dHO}.
\end{proof}

\begin{rem}\label{r:mapofruth}
One may interpret the data of Proposition \ref{propmapofruth} as follows.  Again, thinking of $\delta:E^0\to E^1$, $\delta':E'^0\to E'^1$ as 2-term complexes of vector bundles, the first item says that $\phi^\bullet$ is a chain map.  The second and third items say that $\mu_g$ functions for every $g$ as a homotopy operator between the two composition of chain maps $\phi^\bullet\circ\lambda^\bullet_g$ and $\lambda'^\bullet_g\circ \phi^\bullet$.  Like the fourth item of Proposition \ref{propruth}, the fourth item above will have an interesting interpretation once we have switched to weak actions of groupoids.
\end{rem}

\begin{notation}
From here forward, let $\ruth(G)$ denote the category of 2-term representations up to homotopy of $G$.  We will primarily think of representations up to homotopy and their morphisms as the data detailed in Proposition \ref{propruth} and Proposition \ref{propmapofruth}. 
\end{notation}

Now, we turn to discuss VB-groupoids.  As noted by Gracia-Saz and Mehta \cite{GM}, VB-groupoids may be thought of either as vector bundles in the category of groupoids or as groupoids in the category of vector bundles.  

\begin{defn}
A {\em VB-groupoid over $G$} is a map of Lie groupoids $\pi:V\to G$ such that the maps on objects and morphisms $\pi_0:V_0\to G_0$ and $\pi_1:V_1\to G_1$ are vector bundles and all the structure maps of $V$ are maps of vector bundles.

A {\em map of VB-groupoids over $G$} $f: (\pi:V\to G)\to (\varpi:W\to G)$ is a map of Lie groupoids such that $f_0:V_0\to W_0$ and $f_1:V_1\to W_1$ are maps of vector bundles covering $\id_{G_0}$ and $\id_{G_1}$, respectively.

From here forward, let $\VB(G)$ denote the category of VB-groupoids over $G$.  
\end{defn}

\begin{ex}
For any Lie groupoid $\gpoid{H}{H_1}{H_0}$, the tangent groupoid \\$\gpoid{TH}{TH_1}{TH_0}$ is a VB-groupoid over $H$.
\end{ex}

\begin{ex}
If $\pi:E\to G_0$ is a vector bundle on which $G$ acts, then the action groupoid $\gpoid{G\rtimes E}{G_1\times_{G_0} E}{E}$ is a VB-groupoid over $G$.  
\end{ex}

As shown by Mehta and Gracia-Saz, to any 2-term representation up to homotopy, one may associate a VB-groupoid called {\em the semi-direct product}:

\begin{construction}[(see Example 3.16, \cite{GM})]\label{c:semidirect}
Let $(\delta:E^0\to E^1,\lambda^0,\lambda^1,\Omega)$ be (the data associated to) a representation up to homotopy (see Proposition \ref{propruth}).  Construct a VB-groupoid $\gpoid{V}{V_1}{V_0}$ as follows:
\begin{itemize}
\item For objects, take $V_0=E^1$.
\item For morphisms, take $V_1=s^*E^1\oplus t^*E^0$.
\item For the source and target $\tilde{s},\tilde{t}:V_1\to V_0$ of $V$, define
 \[\tilde{s}(g,e_0,e_1)=e_1\text{ and }\tilde{t}(g,e_0,e_1)=\delta(e_0)+\lambda^1_g(e_1).\]
\item For multiplication, define 
\[m((g,e_0,e_1),(h,f_0,f_1)):= (gh,e_0+\lambda^0_g(f_0)-\Omega_{g,h}(f_1),f_1)\]
\item For units, for each $e\in E^1_x$, $x\in G_0$, take $\tilde{u}(e):=(u(x), 0,e)$.
\item For inverses, take $(g,e_0,e_1)^{-1}:=(g^{-1},-\lambda^0_{g^{-1}}(e_0)+\Omega_{g^{-1},g}(e_1),\delta(e_0)+\lambda^1(e_1))$ 
\end{itemize}
\end{construction}

\begin{thm}\label{t:ruthsandVB}
Let $\Psi:\ruth(G)\to \VB(G)$ be the functor defined as follows:
\begin{itemize}
	\item On objects: let $\Psi(\delta:E^0\to E^1,\lambda^0,\lambda^1,\Omega)$ be the semi-direct product, as built in Construction \ref{c:semidirect}.
	\item On morphisms: let $\Psi(\phi^0,\phi^1,\mu)$ be the morphism of VB-groupoids $(g,e_0,e_1)\mapsto (g,\phi^0(c)+\mu_g(e_1),\phi^1(c))$.
\end{itemize}
Then $\Psi$ is an equivalence of categories. 
\end{thm}

\begin{proof}
In \cite{GM}, Gracia-Saz and Mehta showed that $\Psi$ on objects induces an bijection on isomorphism classes.  In Theorem 2.7 of \cite{dHO}, del Hoyo and Ortiz built the functor $\Psi$ and actually showed this bijection came from an equivalence of categories.
\end{proof}

One tool required in the proof of Theorem \ref{t:ruthsandVB} are connections associated to VB-groupoids.  This is a standard tool used by Arias Abad and Crainic \cite{AC} and Gracia-Saz and Mehta \cite{GM}.  It also has a more general version, known as a cleavage, for general fibrations of groupoids (see for instance Definition 2.1.4, \cite{dHF}).  As we need connections to define the adjoint representation up to homotopy associated to any groupoid and we will also need this tool later, we define it here:

\begin{defn}\label{d:connection}
Let $\pi:V\to G$ be a VB-groupoid over $M$ with source and target maps $\tilde{s},\tilde{t}:V_1\to V_0$ and unit map $\tilde{u}:V_0\to V_1$.  Then a {\em connection on $V$} is a map of vector bundles $\sigma:s^*V_0\to V_1$ that is a splitting of the map $\tilde{s}:V_1\to s^*V_0$ (naturally extended to all of $V_1$).

For each $g\in G_1$, denote by $\sigma_g$ the restriction of $\sigma$ to the fiber $s^*V_0|_g$.

A connection is {\em unital} if, for every unit $u$ of $G$, $\sigma_u(v)=\tilde{u}(v)$.
\end{defn}

\begin{rem}
As explained by Arias Abad and Crainic \cite{AC}, one may find a unital connection for any VB-groupoid.
\end{rem}

\begin{ex}
For $TG$ the tangent groupoid associated to $G$, the representation up to homotopy associated to $TG$ is the adjoint representation on the algebroid $\rho:A\to TG_0$ of $G$.  Explicitly, the adjoint representation was defined by Arias Abad and Crainic as follows (see Section 3.2, \cite{AC}).  First, choose a connection $\sigma$ of $G$.  Then the adjoint representation up to homotopy induced by $\sigma$ is that with quasi-actions 
\[\lambda^1_g(v):=dt(\sigma_g(v))\] on $TG_0$ and 
\[\lambda^0_g(X):= \sigma_g(\rho(X))\cdot X\cdot 0_{g^{-1}}\]
on $A$.  Here, $\cdot$ denote multiplication in $TG$ and $0_{g^{-1}}$ is the zero element of $T_{g^{-1}}G_1$.  As explained by Gracia-Saz and Mehta, $\sigma_g(\rho(X))$ and $0_{g^{-1}}$ are the unique elements of $TG_1$ by which $X$ may be multiplied on the left and right, respectively (see Section 3.3, \cite{GM}).  

As noted by Arias Abad and Crainic (see Proposition 3.16, \cite{AC}), for any two choices of connection $\sigma$, $\sigma'$ of $G$, the corresponding representations up to homotopy are isomorphic.
\end{ex}

\section{2-Term chain complexes and linear Lie groupoid bundles}\label{s:2-termandLLGB}

To establish the equivalence of 2-term representations up to homotopy and linear Lie groupoid bundles, it will be convenient to establish an equivalence of 2-categories between the 2-categories of 2-term chain complexes and linear Lie groupoid bundles.  This is just a scaled up version of a result already proven by Baez and Crans (see Theorem 12, \cite{BC}).  

To begin, let's build the 2-category of linear Lie groupoid bundles: 

\begin{defn}\label{d:VBoverM}
Given a manifold $M$, a {\em linear Lie groupoid bundle} over $M$ is a VB-groupoid over the trivial groupoid associated to $M$ (i.e., the groupoid with objects and morphisms $M$ and structure maps $\id_M$).

A {\em map of linear Lie groupoid bundles} $f:(\pi:V\to M) \to (\varpi:W\to N)$ is a map of Lie groupoids such that $f_0:V_0\to W_0$ and $f_1:V_1\to W_1$ are maps of vector bundles covering the same map on the base ${\bar{f}}:M\to N$.

A {\em natural transformation of linear Lie groupoid bundles} is a natural transformation $\alpha:f\Rightarrow g$ of linear Lie groupoid bundle maps covering $\bar{f}:M\to N$ that is also a map of vector bundles covering $\bar{f}$.  Explicitly, it is a map of vector bundles $\alpha:V_0\to W_1$ covering $\bar{f}:M\to N$ such that, for every arrow $v:x\to y$ in $V$, the diagram
\[\xymatrix{f(x)\ar[r]^{f(v)} \ar[d]_{\alpha(x)} & f(y)\ar[d]^{\alpha(y)} \\ g(x) \ar[r]_{g(v)} & g(y)}\] 
\end{defn}

\begin{rem}
As the name suggests, it is useful to think of a linear groupoid bundle $\pi:V\to M$ as smooth family of linear Lie groupoids.  Indeed, note that for each $x\in M$, we have a linear groupoid $\gpoid{\pi^{-1}(x)}{\pi_1^{-1}(x)}{\pi_0^{-1}(x)}$ (i.e., a groupoid with objects and morphisms vector spaces and all structure groups linear maps) which is a Lie subgroupoid of $V$.  
\end{rem}

We may now define the 2-category of linear Lie groupoid bundles over a particular manifold $M$.

\begin{defn}
Let $\VB$ be the 2-category of linear Lie groupoid bundles defined as follows:

\begin{itemize}
\item Objects: linear Lie groupoid bundles $\pi:V\to M$.

\item 1-morphisms: maps of linear Lie groupoid bundles $f:V\to W$.

\item 2-morphisms: natural transformations of linear Lie groupoid bundles.
\end{itemize}
1-morphisms compose as expected while 2-morphisms compose vertically and horizontally as natural transformations.
\end{defn}

To see this is a well-defined 2-category, we must check vertical and horizontal composition are well-defined.

\begin{prop}
Natural transformations of linear Lie groupoid bundles are closed under vertical and horizontal composition. 
\end{prop}

\begin{proof}
Let $\pi:V\to L$ and $\pi':W\to M$ be linear Lie groupoid bundles and let $f:L\to M$ be a smooth map.  Let $F,F', F'':V\to W$ be three maps of linear Lie groupoid bundles and let $\eta:F\Rightarrow F'$ and $\eta:F'\Rightarrow F''$ be two natural transformations of linear Lie groupoid bundles.  Recall that the vertical composition of $\eta'$ with $\eta$ takes $v\in V_0$ to the arrow $\eta'(v)\circ \eta(v)$.  

Thinking of $\eta$ and $\eta'$ as smooth maps $\eta,\eta':V_0\to W_1$, then, this means that vertical composition $\eta'\circ\eta:V_0\to W_1$ is just the smooth map $\eta'\cdot \eta:V_0\to W_1$, for $\cdot$ multiplication in $W$.  Note that, since $\eta$ and $\eta'$ induce a smooth map $(\eta,\eta'):V_0\to W_1\times_{W_0} W_1$ of vector bundles from $V_0$ to the composable arrows of $W$.  Thus, $\eta'\circ \eta$ factors as maps of vector bundles $m\circ (\eta,\eta')$.  

A very similar argument show that, the horizontal composition $\nu * \eta$ of natural isomorphisms $\eta:F\Rightarrow F'$ and $\nu:G\Rightarrow G'$ which  takes $c\in V_0$ to the composition $\nu(F'(v))\circ G(\eta(v))$ must also be a natural transformation of linear Lie groupoid bundles.  
\end{proof}

Now, we may define the 2-category of 2-term vector bundles over a manifold $M$:

\begin{defn}
	Let $\cchain$ be the 2-category of length 2 cochain complexes defined as follows:
	\begin{itemize}
		\item Objects: length two chain complexes of vector bundles $\xymatrix{C_0\ar[r]^\delta & C_1}$ over $M$ (for which $\delta$ covers the identity map on $M$).
		\item Morphisms: for $C^\bullet$ and $D^\bullet$ chain complexes over manifolds $M$ and $N$, morphisms of $\cchain$ are chain maps $f^\bullet: C^\bullet \to D^\bullet$
		\[\xymatrix{C^0 \ar[d]_{f^0} \ar[r]^{\delta^C} & C^1 \ar[d]^{f^1} \\ D^0 \ar[r]_{\delta^D} & D^1}\]
		such that $f^0$ and $f^1$ both cover a smooth map $f:M\to N$.

		\item 2-morphisms: for $f^\bullet,f'^\bullet:C^\bullet \to D^\bullet$ two morphisms, a 2-morphism is a chain homotopy
		\[\Omega:(f^\bullet:C^\bullet \to D^\bullet)\Rightarrow (g^\bullet : C^\bullet \to D^\bullet).\]  
		Explicitly, this is just a single map $\Omega:C^1\to D^0$ with $\delta^D\circ \Omega = g^1-f^1$ and $\Omega\circ \delta^C = g^0-f^0$.
	\end{itemize}

Composition is defined as follows:
\begin{itemize}
\item 1-morphisms are composed as chain maps; i.e., for $h^\bullet:= g^\bullet \circ f^\bullet$, take $h^i:= g^i\circ f^i$.  

\item 2-morphisms are composed vertically by sum.

\item Horizontally, we have the following operation: suppose we have 2-morphisms $\Psi:(f^\bullet:C^\bullet \to D^\bullet)\Rightarrow (g^\bullet : C^\bullet \to D^\bullet)$ and $\Omega:(k^\bullet:D^\bullet \to E^\bullet)\Rightarrow (\ell^\bullet:  D^\bullet \to E^\bullet)$.  For reference, these fit into the following diagram:
\[\xymatrix@C6em@R4em{C^0 \ar[r]^{\delta^C} \ar@<1.5pt>[d]^{f^0} \ar@<-1.5pt>[d]_{g^0} & C^1 \ar@<1.5pt>[d]^{f^1} \ar@<-1.5pt>[d]_{g^1} \ar[dl]_-\Psi \\ D^0 \ar[r]^{\delta^D} \ar@<1.5pt>[d]^{k^0} \ar@<-1.5pt>[d]_{\ell^0} & D^1 \ar@<1.5pt>[d]^{k^1} \ar@<-1.5pt>[d]_{\ell^1} \ar[dl]_-{\Omega} \\ E^0 \ar[r]_{\delta^E} & E^1}\]
Then the horizontal composition is given by 
\[\Omega'*\Omega:= \left(k^0\circ \Psi + \Omega\circ g^1\right) : (k^\bullet\circ f^\bullet)\Rightarrow (\ell^\bullet \circ g^\bullet)\]
\end{itemize}

For any two cochain complexes $C^\bullet$ and $D^\bullet$, the zero map from $C^1$ to $D^0$ is the identity 2-morphism.
\end{defn}

For completeness, let's check that this is indeed a 2-category.  The only non-trivial fact to check is the interchange law.  So consider the diagram of 2-morphisms

\[\xymatrix@C=8em{ C^\bullet \ar@/^2pc/[r]^(.45){f^\bullet} \ar[r]_(.45){g^\bullet} \ar@/_2pc/[r]_(.45){h^\bullet} \ar@/^2pc/[r]^(.5){}="a" \ar@{}[r]^(.5){}="b" \ar@/^-2pc/[r]^(.5){}="c" \ar@{=>}"a";"b"^\Phi \ar@{=>}"b";"c"^X &  D^\bullet \ar@/^2pc/[r]^(.45){k^\bullet} \ar[r]_(.45){\ell^\bullet} \ar@/_2pc/[r]_(.45){m^\bullet} \ar@/^2pc/[r]^(.5){}="d" \ar@{}[r]^(.5){}="e" \ar@/^-2pc/[r]^(.5){}="f" \ar@{=>}"d";"e"^\Psi \ar@{=>}"e";"f"^\Omega & E^\bullet}\]
Then the calculation
\begin{align*}
\Phi*\Psi+X*\Omega &= (k^0\circ \Phi+\Psi\circ g^1) + (l^0\circ X + \Omega\circ h^1)  \\
				   &= k^0 \circ \Phi + \Psi\circ (g^1-h^1) +\Psi\circ h^1 + (l^0-k^0)\circ X +k^0\circ X +\Omega\circ h^1 \\
				   &= k^0\circ \Phi + \Psi\circ (-(\delta^D\circ X)) +\Psi\circ h^1 +(\Psi\circ \delta^D)\circ X +k^0\circ X +\Omega\circ h^1  \\&= k^0\circ(\Phi+X)+(\Psi+\Omega)\circ h^1 \\ 
				   &= (\Phi+X)*(\Psi+\Omega)
\end{align*}
confirms that the vertical and horizontal compositions defined for $\cchain$ satisfy the interchange law.

The components of 2-term representations up to homotopy and their morphisms correspond to objects, morphisms, and 2-morphisms in $\cchain$:

\begin{lemma}\label{l:ruthtocchain}
Let $(\delta:E^0\to E^1,\lambda^0,\lambda^1,\Omega)$ be (the data corresponding to) a representation up to homotopy.  Then:
\begin{itemize}
\item $\lambda^0,\lambda^1$ induce a 1-morphism of $\cchain$ $\lambda^\bullet: s^*E^\bullet \to E^\bullet$ covering $t:G_1\to G_0$; and
\item $\Omega$ induces a 2-morphism of $\cchain$ $\Omega:p_1^*\lambda^\bullet \circ p_2^*\lambda^\bullet \Rightarrow m^*\lambda^\bullet$ covering $m:G_2\to G_1$.
\end{itemize}
Here, $p_i:G_2\to G_1$ are the natural projection maps and $m:G_2\to G_1$ is the multiplication map.

Similarly, if $(\phi^0,\phi^1,\mu)$ is (the data corresponding to) a map of representations up to homotopy from $(\delta:E^0\to E^1,\lambda^0,\lambda^1,\Omega)$ to $(\delta':E'^0\to E'^1,\lambda'^0,\lambda'^1,\Omega')$, then:
\begin{itemize}
\item $\phi^0,\phi^1$ induces to a 1-morphism of $\cchain$ $\phi^\bullet$ covering $\id_M$; and 
\item $\mu$ corresponds to a 2-morphism of $\cchain$ $\mu:\phi^\bullet \circ \lambda^\bullet \Rightarrow \lambda'^\bullet \circ \phi^\bullet$ covering $t:G_1\to G_0$.   
\end{itemize}
\end{lemma} 

\begin{proof}
This is simple to verify from the conditions established in Propositions \ref{propruth} and \ref{propmapofruth}; see also Remarks \ref{r:ruth} and \ref{r:mapofruth}.
\end{proof}

Now, we may define an equivalence of 2-categories between $\cchain$ and $\VB$.  On objects, we use the following construction.

\begin{construction}\label{const:sumgpoid}
Let $\delta:C^0\to C^1$ be a cochain complex over the manifold $M$.  Then we may build a VB-groupoid $V$ over $M$ as follows:
\begin{itemize}
\item The objects of $V$ are just the vector bundle $C^1$.
\item The morphisms of $V$ are the vector bundle $C^0\oplus C^1$.
\item The source map is the projection $p:C^0\oplus C^1\to C^1$ while the target map is the sum $\delta+p:C^0\oplus C^1 \to C^1$.
\item Multiplication is given as $(c_0,c_1)\cdot (c'_0,c'_1):= (c_0+c_0',c_1')$ (indeed, this choice of multiplication is more or less forced by the choice of source and target maps given above).
\item The identity map $u:C^1\to C^0\oplus C^1$ is the inclusion $u(c):=(0,c)$.
\item The inversion map $i:C^0\oplus C^1 \to C^0\oplus C^1$ is the map $i(c_0,c_1):=(-c_0,\delta(c_0)+c_1)$. 
\end{itemize}
\end{construction}

This construction extends to a strict 2-functor between $\cchain$ and $\VB$:

\begin{prop}\label{p:2-chainsandlLgbs}
Let $\Phi:\cchain\to \VB$ be the following (strict) 2-functor: 
\begin{itemize}
\item Objects: For $C^\bullet:=\left\{ \delta: C^0 \to C^1\right\}$, let $\Phi(C^\bullet)$ be the VB-groupoid \\$\gpoid{V}{C^1\oplus C^0}{C^1}$ built in Construction \ref{const:sumgpoid}.
\item 1-morphisms: given a chain map $f^\bullet:C^\bullet \to D^\bullet$ over a morphism $f:M\to N$, take $\Phi(f^\bullet)$ to be the functor which on objects is just the map $f^1$ and on morphisms is the map $f^0+f^1:C^0\oplus C^1 \to D^0 \oplus D^1$.
\item 2-morphisms: given a chain homotopy $\Omega:f^\bullet \Rightarrow g^\bullet$ between maps $f^\bullet, g^\bullet:C^\bullet \to D^\bullet$, let $\Phi(\Omega)$ be the natural transformation given by the map
\[
\Phi(\Omega):C^1 \longrightarrow D^0\oplus D^1,\;\;\;
\Phi(\Omega)(c):= (f^1(c),\Omega(c))
\]  
\end{itemize}
Then $\Phi$ is an equivalence of 2-categories.  In fact, for every $C^\bullet$, $D^\bullet\in \cchain(M)$, $\Phi$ induces an isomorphism of categories
$\Phi:\Hom(C^\bullet,D^\bullet)\to \Hom(\Phi(C^\bullet),\Phi(D^\bullet))$. 
\end{prop}

\begin{proof}
Recall that equivalences of 2-categories is an essentially surjective 2-functor that induces an equivalence of categories on $\Hom$ categories.

 To begin, let's show $\Phi$ is essentially surjective.  Suppose that $\pi:V\to M$ is a linear Lie groupoid bundle.  Let $\tilde{s},\tilde{t}:V_1\to V_0$ be the source and target maps of $V$ and let $\tilde{u}:V_0\to V_1$ be the unit map.  Define $C^1:=V_0$ and $C^0:=\ker(\tilde{s})$.  Note that $C^0$ is a well-defined subbundle of $V_1$ since $\tilde{s}$ must be full rank everywhere (indeed, it is a submersion and a map of vector bundles). 

 Next, note that $V_1$ splits as $C^0\oplus \tilde{u}(V_0)\cong C^0\oplus C^1$ and, with respect to this splitting, $\tilde{s},\tilde{t}|_{C^1}\equiv \mathrm{id}_{C^1}$.  Then for $\delta:=\tilde{t}|_{C^0}$ and $p:C^0\oplus C^1\to C^1$ the projection, $\tilde{s}=p$ and $\tilde{t}=\delta+p$.  

 Now, suppose $(c_0,c_1)$ and $(c'_0,c'_1)$ are composable arrows in $V_1$.  It follows that $t(c'_0,c'_1)=c'_1+\delta(c'_0)=c_1=s(c_0,c_1)$.  The product $(c_0,c_1)\cdot(c'_0,c'_1)$ splits as
 \[(c_0,c_1)\cdot(c'_0,c'_1)=(c_0, c_1')\cdot (0,c'_1)+(0, \delta(c_0))\cdot (c_0',0)\]
 Since $(0,c_1')=u(c_1')$ and $(0,\delta(c_0))=u(\delta(c_0))$, it follows that 
 \[(c_0,c_1)\cdot(c'_0,c'_1) = (c_0,c_1')+(c_0',0)=(c_0+c_0',c_1').\]
 Thus, the splitting above induces a smooth, invertible functor from $\Phi(\{\delta: C^0\to C^1\})$ to $V$.   

Let $C^\bullet=\{\delta^C: C^0 \to C^1\}$ and $D^\bullet=\{\delta^D:D^0\to D^1\}$ be two objects of $\cchain$.  Let $F:\Phi(C^\bullet)\to (D^\bullet)$ be any map of Lie groupoids in $\VB$.  Then the map on morphisms $F_1:C^0\oplus C^1 \to D^0\oplus D^1$ splits into four maps of vector bundles:
\begin{align*}
w:C^0 \to D^0 \;\;&\;\; x: C^0\to D^1 \\
y:C^1\to D^0 \;\;&\;\; z: C^1 \to D^1
\end{align*} 
that is, $F_1=(w+y)\oplus(x+z)$.  

Suppose that $(c_0,c_1)$ and $(c_0',c_1')$ are composable arrows in $\Phi(C^\bullet)$.  As $F$ is a functor, it must preserve multiplication.  On the one hand, we have
\[F((c_0,c_1)\cdot (c_0',c_1')) = F(c_0+c_0',c_1') = (w(c_0)+w(c_0')+y(c_1'),x(c_0)+x(c_0')+z(c_1'))\]
On the other, we have
\begin{align*}F(c_0,c_1)\cdot F(c_0',c_1') &= (w(c_0)+y(c_1),x(c_0)+z(c_1))\cdot (w(c_0')+y(c_1'),x(c_0')+z(c_1')) \\ &= (w(c_0)+y(c_1)+ w(c_0')+y(c_1'), x(c_0')+z(c_1'))\end{align*}
As this must hold for every pair of composable arrows, we must conclude $x=0$ and $y=0$. 

It is easy to check that, since $F$ must preserve source maps, $z=F_0$ (i.e., $F$ on objects) and, since $F$ must preserve target maps, $w:C^0\to D^0$ and $z:C^1\to D^1$ yield a chain map between $C^\bullet$ and $D^\bullet$.

Finally, note that it is clear that $\Phi$ induces an isomorphism on 2-morphisms; indeed, $(f^1,\Omega)$ is uniquely determined by $\Omega$.  If $\alpha:\Phi(f^\bullet) \Rightarrow \Phi(g^\bullet)$ is a natural transformation, then for any $c\in \Phi(C^\bullet)_0=C^1$, $\alpha(c)$ is an arrow with source $f^1(c)$ and target $g^1(c)$.  Thus, the arrow $\alpha(c)=(f^1(c),\bar{\alpha}(c))$ where $f^1(c)+\delta^D(\bar{\alpha}(c))=g^1(c)$.  Similarly, one may check that $\bar{\alpha}(\delta^C(c))=g^0(c)-f^0(c)$.  Thus, $\bar{\alpha}:f^\bullet\to g^\bullet$ is a chain homotopy.            
\end{proof}

\section{Weak representations of groupoids}\label{s:WReps}

In this section, we will define weak representations of a groupoid $G$.  We will also show that the isomorphism $\Phi:\cchain \to \VB$ given in Proposition \ref{p:2-chainsandlLgbs} induces an isomorphism between the categories of representations up to homotopy and weak representations.  

Weak representations are, in particular, weak actions of groupoids; the definition of a weak action is a repurposed version of Burszytyn, Noseda, and Zhu's definition of the action of a stacky Lie groupoid on a stack (see Definition 3.15, \cite{BNZ}).

\begin{defn}\label{d:waction}
Let $H$ be a Lie groupoid and $f:H\to G_0$ be a map of Lie groupoids.  Then a {\em weak action of $G$ on $H$} is
\begin{itemize}
\item a smooth functor $A:G\times_{G_0} H\to H$;
\item a smooth natural isomorphism 
\[\xymatrix@C8em{G_2\times_{G_0} H \ar@/^1pc/[r]^(0.5){}="a" \ar@/^1pc/[r]^{A\circ (\id_G\times A)} \ar@/_1pc/[r]_(0.5){}="b" \ar@/_1pc/[r]_{A\circ (m\times \id_H)} \ar@{=>}^\alpha"a";"b" & H}\]
i.e., a smooth collection of arrows of $H$ such that, for every morphism $(g,k,h):(g,k,x)\to (g,k,y)\in G_2\times_{G_0}H$, the diagram
\[\xymatrix{g\cdot k \cdot x \ar[d]_{\alpha(g,k,x)} \ar[r]^{g\cdot k\cdot h} & g\cdot k \cdot y \ar[d]^{\alpha(g,k,y)} \\ (gk)\cdot x \ar[r]_{(gk)\cdot h} & (gk)\cdot y}\]
commutes (where $g\cdot x:=A(g,x)$); and
\item a smooth natural isomorphism 
\[\xymatrix@C8em{H \ar@/^1pc/[r]^(0.5){}="a" \ar@/^1pc/[r]^{A\circ (u \circ f)} \ar@/_1pc/[r]_(0.5){}="b" \ar@/_1pc/[r]_{\id_H} \ar@{=>}^\varepsilon"a";"b" & H}\]
i.e., $\varepsilon$ naturally relates the action of units with the identity functor of $H$.
\end{itemize}  

$A$, $\alpha$, and $\varepsilon$ are subject to the following coherence conditions:
\begin{itemize}
\item $f\circ A = t\circ p$, for $p$ the natural projection $p:G\times_{G_0} H\to G$.
\item For all appropriate choices of $(g,k,l)\in G_3$ and $x\in H_0$, $\alpha:G_2\times_{G_0}H_0\to H_1$ satisfies 
\[\alpha(g,kl,x)\cdot g\cdot \alpha(k,l,x) = \alpha(gk,l,x)\alpha(g,k,l \cdot x)\]
More concretely, we require that the diagram
\begin{equation}\label{eq:alphacoh}\xymatrix{& g\cdot k\cdot l\cdot x \ar[dl]_{g\cdot \alpha(k,l,x)} \ar[dr]^{\alpha(g,k,l\cdot x)} & \\g\cdot (kl)\cdot x \ar[dr]_{\alpha(g,kl,x)} & & (gk)\cdot l \cdot x \ar[dl]^{\alpha(gk,l,x)} \\ & (gkl)\cdot x &}\end{equation} 
commutes.
\item For any $(g, x)\in G_1\times_{G_0} H_0$, $\alpha(g,u(f(x)),x)=g\cdot \epsilon(x)$ and $\alpha(u(t(g)),g,x)=\varepsilon(g\cdot x)$.

\end{itemize} 

A weak action of $G$ on $H$ is {\em unital} if $\varepsilon=\id_{\id_H}$.
\end{defn}

\begin{rem}
Let's unpack Definition \ref{d:waction} a bit.  The fact that $A:G_1\times_{G_0} H \to H$ is a functor subsumes two interesting properties.  First note that since $f:H\to G_0$ is a map of Lie groupoids, if $h_1$ and $h_2$ are composable arrows then $f_1(h_1)=f_1(h_2)$.  In particular, if $f_1(h_1)=s_G(g)$, then we have that $A_1(g, h_1\cdot h_2)=A(g,h_1)\cdot A(g,h_2)$.

Next, note that any unit $u_H(x)\in H_1$ and $g\in G_1$ satisfying $f_0(x)=s_G(g)$, we have that $A(g,u_H(x))=u_H(A(g,x))$.

If the action is unital, then $\alpha(g,u(f(x)),x)=\alpha(u(t(g)),g,x)=\id_x$ for all $(g,x)\in G_1\times_{G_0} H_0$.   

Finally, suppose additionally that $f_0,f_1$ are submersions.  Then for every $a\in G_0$, $f^{-1}(a)$ is a Lie subgroupoid.  Furthermore, each $g:a\to b$ in $G$ induces a map of Lie groupoids $A(g,\cdot):f^{-1}(a) \to f^{-1}(b)$.  In the case where $A$ is a {\em non-weak} action, note this functor is an isomorphism of categories with inverse functor $A(g^{-1},\cdot)$.  In the case of a weak action, $A(g,\cdot)$ and $A(g^{-1},\cdot)$ form an equivalence of categories.  Indeed, $A(g^{-1},\cdot)\circ A(g,\cdot)$ and $\id_{f^{-1}(a)}$ are naturally isomorphic via the smooth natural isomorphism 
\[\varepsilon\circ \alpha(g^{-1},g,\cdot) : A(g^{-1},\cdot)\circ A(g,\cdot) \Rightarrow \id_{f^{-1}(a))}, \;\;\; x\mapsto \varepsilon(x)\cdot \alpha(g^{-1},g,x) \]   
\end{rem}

To define equivariant maps, we again refer to the work of Bruszytyn, Noseda, and Zhu (see Definition 3.20, \cite{BNZ}):

\begin{defn}\label{d:equivmap}
Let $f:H\to G_0$ and $f':H'\to G_0$ be two Lie groupoid maps to $G_0$ and let $(A,\alpha,\varepsilon)$ and $(A',\alpha',\varepsilon')$ be weak actions of $G$ on $H$ and $H'$, respectively.  Then a map of Lie groupoids $F:H\to H'$ is {\em equivariant} if $f'\circ F=f$ and if there exists there exists a natural isomorphism $\delta:F\circ A\Rightarrow A'\circ (\id_{G_1} \times F)$ such that, for all $(g,k,x)\in G_2\times_{G_0}H_0$, the diagrams
\begin{equation}\xymatrix@C=0.5em@R3em{& & F(g\cdot k \cdot x) \ar[dll]_{\delta(g,k\cdot x)} \ar[drr]^{F(\alpha(g,k,x))} & & \\ g\cdot F(k\cdot x)\ar[dr]_{g\cdot \delta(k,x))} & & & & F((gk)\cdot x) \ar[dl]^{\delta(gk,x)} \\ & g\cdot k \cdot F(h) \ar[rr]_{\alpha'(g,k,F(x))} & & (gk)\cdot F(x)  }\end{equation}
and
\[\xymatrix@C7em@R5em{F((f(x))\cdot x) \ar[r]^-{F(\varepsilon(x))} \ar[d]_{\delta(u(f(x)), x)} & F(x) \\ u(f(x))\cdot F(x) \ar[ur]_{\varepsilon'(F(x))} }\]
 
commute.

\end{defn}

\begin{prop}
The composition $(F'\circ F, (\id_{F'}\circ \delta)*(\delta'\circ \id_{F\times \id}))$ of two equivariant maps $(F,\delta)$ and $(F',\delta')$ is again an equivariant map.
\end{prop}

\begin{proof}
See Lemma 3.21, \cite{BNZ}.
\end{proof}

Now, we may define weak representations of a groupoid $G$.  As is the case with groupoid actions, weak representations are groupoid actions on smooth linear spaces.  In this case, we use linear Lie groupoid bundles:

\begin{defn}
A {\em weak representation of $G$} on a linear groupoid bundle $\pi:V\to G_0$ is a weak unital action $(A,\alpha,\id_V)$ such that $A:G_1\times_{G_0}V\to V$ is a map of linear Lie groupoid bundles covering $t:G_1\to G_0$ and $\alpha$ is a natural transformation of linear groupoids covering $t_2:G_2\to G_0$.

A {\em map of weak representations of $G$} is any map of linear groupoid bundles which is also equivariant.

From here forward, denote by $\WR(G)$ the category of weak representations of $G$.  
\end{defn}

\begin{rem}
It is a well known fact that a Lie groupoid representation $G$ on a vector bundle $E\to G_0$ corresponds exactly to a smooth functor from $G$ to the frame groupoid of $E$: the Lie groupoid with objects $G_0$ and morphisms linear isomorphisms between the fibers of $E$.  

We may make a similar statement for weak representations.  Let $\pi:V\to G_0$ be a linear Lie groupoid bundle.  Let ${\sf Fr}(V)$ denote the 2-category with objects $G_0$, morphisms linear equivalences of categories $\varphi:\pi^{-1}(x)\to \pi^{-1}(y)$, and 2-morphisms linear natural isomorphisms.  It is easy to check then that a weak representation corresponds to a pseudofunctor from $G$ (thought of as a 2-category with 2-morphisms $G_2$, the collection of composable morphisms) to ${\sf Fr}(V)$.  

Going from pseudofunctors $F:G\to {\sf Fr}(V)$ to weak representations is more challenging.  One should expect that it is smooth pseudofunctors that should correspond to weak representations.  However, ${\sf Fr}(V)$ needn't be a Lie 2-groupoid; indeed, one may construct examples where the source and target maps associated to ${\sf Fr}(V)$ will fail to be submersion for any possible manifold structure on the morphisms of ${\sf Fr}(V)$.

One possible way around this is the use of diffeologies; this approach may be explored in future work.
\end{rem}

Now, we may prove the equivalence between the categories of 2-term representations up to homotopy of $G$ and weak representations of $G$.

\begin{thm}\label{t:ruthsandwr}
The equivalence of 2-categories $\Phi:\cchain\to \VB$ induces an equivalence of categories $\Phi(G):\ruth(G)\to \WR(G)$.
\end{thm}

\begin{proof}
Define $\Phi(G)(\delta:E^0\to E^1,\lambda^0,\lambda^1,\Omega):=(\Phi(\delta:E^0\to E^1), \Phi(\lambda^\bullet), \Phi(\Omega))$; here we are using Lemma \ref{l:ruthtocchain} to identify the components associated to $(\delta:E^0\to E^1,\lambda^0,\lambda^1,\Omega)$ with elements of $\cchain$.  Similarly, define $\Phi(G)(\phi^0,\phi^1,\mu):=(\Phi(\phi^\bullet), \Phi(\mu))$.  It is easy to verify that the conditions of Propositions \ref{propruth} and \ref{propmapofruth} (i.e., on the data associated to representations up to homotopy and their maps) correspond exactly to the conditions on weak actions and equivariant maps (see Definitions \ref{d:waction} and \ref{d:equivmap}).   Thus, $\Phi(G)$ is well defined and, since $\Phi$ is essentially surjective and induces an isomorphism on $\Hom$ categories, $\Phi(G)$ must be essentially surjective and fully faithful.
\end{proof}

\section{VB-groupoids as weak action groupoids}\label{s:VBsandWR}

In this section, we discuss the action groupoids associated to weak groupoid actions.  We will also explain how VB-groupoids may be realized as action groupoids associated to weak representations.  

\begin{construction}
Let $f:H\to G_0$ be a map of Lie groupoids $G_0$ and let $(A,\alpha, \varepsilon)$ be a weak action of $G$ on $H$.  Write $\tilde{s},\tilde{t}:H_1\to H_0$ and $\tilde{u}:H_0\to H_1$ for the source, target, and unit maps of $H$.  Then {\em the action groupoid $G\times_A H$} is the groupoid with

\begin{itemize}\label{c:actiongpoid}
	\item Objects: $H_0$
	\item Morphisms: For $A_0:G_1\times_{G_0}H_0\to H_0$ the action functor $A$ on objects, the morphisms of $G\times_A H$ are 
	\begin{equation*} \left(G_1\times_{s,{G_0},f_0}H_0\right) \times_{A_0,H_0,\tilde{t}}H_1 =\{(g,x,h)\in G_1\times H_0\times H_1 \;| \; s(g)=\pi_0(x),\; \tilde{t}(h)=g\cdot x\} \end{equation*} 
	\item Source and target: the source map $ \sigma: (G\times_A H)_1 \to H_0$ is the projection $\sigma(g,x,h):=x$ while the target map $\tau:(G\times_AH)_1\to H_0$ is the map $\tau(g,x,h):=s(h)$.
	\item Multiplication: multiplication is given by 
	\[(g,x,h)\cdot (g',x',h'):= (gg', x',\alpha(g,g',x')(g\cdot h') h)\]
	\item Units: the unit map is given by $\mu(x):= (u(f(x)), x, \tilde{u}(x))$.
	\item Inverses: define $(g,x,h)^{-1}:=(g^{-1}, s(h), (g^{-1}\cdot h)^{-1} \alpha(g^{-1},g,x)^{-1} \varepsilon(x)^{-1})$     
\end{itemize}

For the sake of completeness, let's verify $G\times_A H$ is indeed a Lie groupoid.  First, note that the objects and morphisms indeed must be manifolds; on objects, we are already using $H_0$ and on morphisms, we are pulling back against submersions $s:G_1\to G_0$ and $\tilde{t}:H_1\to H_0$.  As projections, $\sigma$ and $\tau$ are certainly submersions.  As all structure maps are smooth combinations of the structure maps of $G$ and $H$, it follows that the structure maps of $G \times_A H$ are smooth as well.

To see that multiplication is associative, note that:
\begin{align*}
((g,x,h) (g',x',h'))  (g'',x'',h'') &= (gg'g'', x'', \alpha(gg',g'',x'')((gg')\cdot h'')\alpha(g,g',x')(g\cdot h')h) \\
&= (gg'g'',x'', \alpha(gg',g'',x'')\alpha(g,g',g''\cdot x'') (g\cdot g'\cdot h'')(g\cdot h')h)\\
&= (gg'g'',x'', \alpha(g,g'g'',x'')(g\cdot \alpha(g',g'',x'')) (g\cdot g'\cdot h'')(g\cdot h')h)\\
&= (gg'g'',x'', \alpha(g,g',g'',x'') (g\cdot(\alpha(g',g'',x'')(g'\cdot h'')h'h)))\\
&= (g,x,h) ((g',x',h')  (g'',x'',h''))
\end{align*}
Here, we take advantage of the fact that $\alpha$ is a natural transformation, the coherence condition \eqref{eq:alphacoh} on $\alpha$, and the functoriality of $A$.  

Using a similar argument also involving the compatibility conditions on $\alpha$ and between $\alpha$ and $\varepsilon$ (see Definition \ref{d:waction}), one may also check that the units and inverses defined above behave as required. 
\end{construction}

\begin{rem}
As noted above, weak groupoid actions are inspired by the definition of a stacky Lie groupoid actions of Bursztyn, Noseda, and Zhu \cite{BNZ}.  Similarly, the construction of an action groupoid associated to a weak groupoid action is also inspired by their work.  As they note in Corollary 5.3, under certain conditions, the map of stacks associated to a stacky Lie groupoid action is a bibundle inheriting a Lie groupoid structure.  

Note that the space of morphisms of $G\times_A H$ in Construction \ref{c:actiongpoid} is exactly the (total space of) the bibundle $\langle A \rangle$ associated to the action functor $A$.  Multiplication in $G\times_AH$ is also related to the bibundle picture; indeed, one may identify the collection of composable arrows of $G\times_AH$ with the bibundle $\langle A\circ (\id_G\times A)\rangle$.  Then $\alpha$ induces an isomorphism of bibundles 
\[\langle \alpha \rangle:\langle A\circ (\id_G\times A)\rangle\to \langle A\circ (m\times \id_H)\rangle\]
and one may check that, on the arrow component $H_1$ of $G \times_A H$, the multiplication defined above corresponds to the composition of $\langle \alpha \rangle$ with projection to $H_1$.  
\end{rem}

Action groupoids of weak representations are naturally VB-groupoids.

\begin{lemma}\label{l:WRactiongpoids}
Let $\pi:V\to G_0$ be a linear Lie groupoid bundle and let $(A:G_1\times_{G_0}V\to V, \alpha)$ be a weak representation of $G$ on $V$.  Then the action groupoid $G\times_A V$ is a VB-groupoid.  Furthermore, for $(B:G_1\times_{G_0}W\to W, \beta)$ another weak representation of $G$ on a linear Lie groupoid bundle $\varpi:W\to G_0$, any map of weak representations 
\[(F,\delta): (A:G_1\times_{G_0}V\to V, \alpha) \to (A:G_1\times_{G_0}V\to V, \alpha)\]
extends to a map of VB-groupoids $\act(F,\delta) : G\times_A V \to G\times_BW$.
\end{lemma}

\begin{proof}
By assumption, we have that $f_0:V_0\to G_0$ is a vector bundle.  Note that the space of arrows $\left(G_1\times_{s,{G_0},\pi_0}V_0\right) \times_{A_0,V_0,\tilde{t}}V_1$ must also naturally inherit the structure of a vector bundle.  Indeed, as explained in the proof of Proposition \ref{p:2-chainsandlLgbs}, for $C:=\ker\left(\tilde{t}\right)$, we may identify $V_1$ with $V_0\oplus C$ and, with respect to this identification, $\tilde{t}$ corresponds to the natural projection $p:V_0\oplus C \to V_0$ while $\tilde{s}$ corresponds to $p+\tilde{s}|_{C}$.  It follows that, using this identification, we may identify $\left(G_1\times_{s,{G_0},\pi_0}V_0\right) \times_{A_0,V_0,\tilde{t}}V_1$ with the bundle $s^*V_0\oplus t^*C\to G_1$.

As the structure maps of $G\times_AV$ are all built from maps of vector bundles, it follows that $G\times_A V$ is a VB-groupoid.

For $(F,\delta): (A:G_1\times_{G_0}V\to V, \alpha) \to (A:G_1\times_{G_0}V\to V, \alpha)$ a map of weak representations, define $\act(F,\delta)$ as follows:
\begin{enumerate}
\item On objects: for $x\in V_0$, define $\act(F,\delta)(x):=F_0$
\item On morphisms: for $(g,x,v)\in (G\times_A V)_1$, define 
\[\act(F,\delta)(g,x,v):= (g,F(x),\delta(g,x)F(v))\]
\end{enumerate}
As $F$ and $\delta$ correspond to maps of vector bundles, it follows that $\act(F,\delta)$ corresponds to a pair of maps of vector bundles.  

To check that $\act(F,\delta)$ preserves multiplication, let $(g,x,v)$ and $(g',x',v')$ be a pair of composable arrows.  Then:
\begin{align*}
\act(F,\delta)(g,x,v)\act(F,\delta)(g',x',v') &= (gg',F(x'), \beta(g,g',F(x'))(g\cdot (\delta(g',x')F(v')))\delta(g,x)F(v))\\
&= (gg',F(x'), \beta(g,g',F(x'))(g\cdot(g',x'))\delta(g,g'\cdot x')F(g\cdot v')F(v)) \\ 
&= (gg',F(x'),  \delta(gg',x')F(\alpha(g,g',x'))F(g\cdot v')F(v)) \\
&= \act(F,\delta)((g,x,v)(g',x',v'))
\end{align*}
One may use similar arguments to show that $\act(F,\delta)$ must preserve identities.
\end{proof}

\begin{prop}
$\act:\WR(G) \to \VB(G)$ given by
\begin{align*}\act((F,\delta):(V,A,\alpha)\to (W,B,\beta):=\act(F,\delta):G\times_AV\to G\times_BW\end{align*}
(for $\act(F,\delta)$ defined in Lemma \ref{l:WRactiongpoids}) is a well-defined functor.
\end{prop}

\begin{proof}
By Lemma \ref{l:WRactiongpoids}, we have that $G\times_A V$ is a VB-groupoid for any weak representation $(A,\alpha)$ on $V\to G_0$ and that $\act(F,\delta)$ is a map of VB-groupoids for any map of weak representations.

So all that remains to be seen is that $\act$ is a well-defined as a functor, note it is clear that $\act(\id_V,\id_{\id\times F})$ clearly corresponds to the identity functor on $G\times_A V$.  For any two maps of weak representations $(F,\delta):(V,A,\alpha)\to (W,B,\beta)$ and $(G,\epsilon):(W,B,\beta)\to (X,C,\gamma)$ and for any $(g,x,v)\in B \times_A V$,
\begin{align*}
	\act(G,\epsilon)(\act(F,\delta)(g,x,v))&=(g,G(F(x)), \epsilon(g,F(x))G(\delta(g,x)F(v)))\\
											&= (g,G(F(x)), (\id_G \circ \delta)*(\epsilon \circ \id_{F\times \id}))(g,x) G(F(v)))\\
											&= \act((G,\epsilon)\circ(F,\delta))(g,x,v)
\end{align*}
\end{proof}

We may finally state the main theorem of this paper:

\begin{thm}\label{t:main}
The functor $\act:\WR(G)\to \VB(G)$ is an equivalence of categories.  
\end{thm}

To prove this, we work in stages.  First, let's prove essential surjectivity.  Our technique will be to show that the quasi-action defined on the kernel of a VB-groupoid (as defined in Section 2.1 of \cite{dHF}, for instance) is in fact a weak representation.

\begin{lemma}\label{l:esssurj}
For any VB-groupoid $\pi:V\to G$, there exists a weak representation $(A,\alpha)$ of $G$ on a linear Lie groupoid bundle $\varpi:K\to G_0$ such that $G\times_AK\cong V$.
\end{lemma}

\begin{proof}
Let $\tilde{s},\tilde{t}:V_1\to V_0$ be the structure maps of $V$.  Let $K$ be the linear Lie groupoid bundle with space of objects $V_0$ and space of morphisms $V_1|_{G_0}$.  Note this is the kernel of the Lie groupoid fibration $\pi$ (see Definition 2.1.3 of \cite{dHF}).  Choose a unital connection $\sigma:s^*V_0\to V_1$ of $V$ (see Definition \ref{d:connection}).  Define $A:G_1\times_{G_0} K\to K$ as follows:
\begin{itemize}
	\item for $x\in K_0=V_0$, define $A_0(g,x):=\tilde{t}(\sigma_g(v))$; and
	\item for $k\in K_1$, define $A_1(g,k):= \sigma_g(\tilde{t}(k))k (\sigma_g(\tilde{s}(k)))^{-1}$.
\end{itemize}
This is a well-defined functor, as, if $\tilde{s}(k)=x$, $\tilde{t}(k)=y$, then:
\[\tilde{s}(g\cdot k) = \tilde{s}((\sigma_g(x))^{-1})=\tilde{t}(\sigma_g(x))=g\cdot x \text{ and }\tilde{t}(g\cdot k)=\tilde{t}(\sigma_g(y))=g\cdot y.\] 
  For composable $k,k'\in K_1$,
 \begin{align*}g\cdot (kk') &= \sigma_g(\tilde{t}(kk'))kk' \sigma_g(\tilde{s}(kk'))^{-1} = \sigma_g(\tilde{t}(k))k\sigma_g(\tilde{s}(k))^{-1}\sigma_g(\tilde{s}(k)) k'\sigma_g(\tilde{s}(k'))^{-1} \\ &= \sigma_g(\tilde{t}(k))k\sigma_g(\tilde{s}(k))^{-1}\sigma_g(\tilde{t}(k')) k'\sigma_g(\tilde{s}(k'))^{-1} = (g\cdot k)(g\cdot k')   \end{align*}
As the composition of smooth maps of vector bundles, clearly $A_0$ and $A_1$ are both maps of vector bundles.

The fact that $\sigma$ is unital implies that this action must be unital as well.  

Next, define for each $(g,g',x)\in G_2\times_{G_0}V_0$ define 
\[\alpha(g,g',x):=\sigma_{gg'}(x)\sigma_{g'}(x)^{-1}\sigma_g(g'\cdot x)^{-1}.\]
Again, as a combinations of smooth maps of vector bundles, this is again a smooth map of vector bundles.  It is easy to check that $\alpha$ gives the natural transformation $A\circ (\id_G\times A)\Rightarrow A\circ (m\times \id_V)$ and that $\alpha$ satisfies the correct coherence conditions.

Define an isomorphism of VB-groupoids as follows: take $\phi:V \to G\times_AK$ by\\ $\phi_0=\id_{V_0}:V_0\to V_0$ and let
\[\phi_1:(G\times_AK)_1,\;(g,x,k)\mapsto k^{-1}\sigma_g(x)\]
Since $(g,x,k)$ is an arrow from $x$ to $s(k)$, it follows that $\phi_1$ preserves source and target maps.  It is a straightforward calculation to confirm that it also preserves multiplication and units.  As $\phi_1$ has a trivial kernel, it follows that the functor $\phi$ which, on objects is given by $\phi_0$ and on morphisms by $\phi_1$ yields an isomorphism of VB-groupoids over $G$.   
\end{proof}

Now, we prove that $\act$ is fully faithful.

\begin{lemma}\label{l:ff}
The functor $\act:\WR(G)\to \VB(G)$ is fully faithful.
\end{lemma}

\begin{proof}
Fix two weak representations $(A,\alpha)$ and $(B,\beta)$ on linear Lie groupoid bundles $\pi:V\to G_0$ and $\varpi:W\to G_0$, respectively.  To see  that $\act(F,\delta)$ uniquely determines both $F$ and $\delta$, first note that we may naturally identify $V$ and $W$ as VB-Lie subgroupoids of the action groupoids $G\times_AV$ and $G\times_BW$.  Indeed, for $(G\times_AV)$ the full subgroupoid with morphisms $(G\times_AV)|_{V_0}$, there an isomorphism of of VB-groupoids $\iota_V:V\to (G\times_AV)|_{V_0}$ given on objects by $\id_{V_0}$ and on morphisms by
\[\iota_V:V_1\to \left((G\times_AV)|_{V_0}\right)_1,\; v\mapsto (u(\pi(v)),\tilde{s}(v),v^{-1})\]
Similarly, let $\iota_W:W\to (G\times_BW)|_{W_0}$ be an isomorphism from $W$ to the kernel subgroupoid of $G\times_BW$.

It follows then that $F$ uniquely determines $\phi|_{V_0}$ for any $\phi=\act(F,\delta)$.  Now, note that for any $(g,x,v)\in (G\times_AV)_1$, we have that 
\begin{equation}\label{fact}(g,x,v)=(u(\pi(v)),\tilde{t}(v),v)(g,x,\tilde{u}(g\cdot x)).\end{equation}
Additionally, recall that $\act(F,\delta)(g,x,v):=(g,F(x),\delta(g,x)v)$; thus, for $p:(G\times_BW)_1\to W_1$ the natural projection and for $(g,x)\in G_1\times_{G_0}V_0$, take $p(\phi_1(u(\pi(v)),\tilde{t}(v),v))=w$.  Then we have that 
\begin{align*} (g,F(x),\delta(g,x)v) = \phi_1(g,x,v)&=\phi_1(u(\pi(v)),\tilde{t}(v),v)\phi_1(g,x,\tilde{u}(g\cdot x)) \\  &= (u(\pi(v)),\phi_0(\tilde{t}(v)),w)(g,\phi_0(x),\tilde{u}(\phi_0(g\cdot x))) \\ &= (g,F(x),w)  \end{align*}
Thus, $\delta(g,x)$ is also uniquely determined by $\phi$ and so $\act$ is faithful.  

To see it is full, we may reverse this process. Indeed, let $\phi:G\times_AV\to G\times_BW$ be a map of VB-groupoids.  Then since any map of VB-groupoids $\phi$ must take $(G\times_AV)|_{V_0}$ to $(G\times_BW)|_{V_0}$, it follows that $\iota_W^{-1}\circ \phi\circ \iota_V$ gives a well-defined functor $F$ from $V$ to $W$.  Note this functor is explicitly given by $F_0=\phi_0$ and $F_1=\tilde{i}\circ p\circ \phi_1$ (for $\tilde{i}:W_1\to W_1$ inversion in $W$).

Again using the decomposition \eqref{fact} above, we have that
\[\phi_1(g,x,v)=\phi_1(u(\pi(v)),\tilde{t}(v),v)\phi_1(g,x,\tilde{u}(g\cdot x)) \]
So, define 
\[\delta(g,x)=p(\phi_1(g,x,\tilde{u}(g\cdot x))).\]
It follows then by definition that $\act(F,\delta)=\phi$.  It is not difficult to show that this $\delta$ must satisfy the necessary conditions for $(F,\delta)$ to be a map of weak representations.
\end{proof}

\begin{proof}[Proof of Theorem \ref{t:main}]
By Lemma \ref{l:esssurj}, $\act$ is essentially surjective and by Lemma \ref{l:ff}, $\act$ is fully faithful.  Therefore, it defines an equivalence of categories.   
\end{proof}

\begin{rem}
For $\Psi:\ruth(G)\to \VB(G)$ the equivalence of categories of Theorem \ref{t:ruthsandVB} and $\Phi(G):\ruth(G)\to \WR(G)$ the equivalence of categories of Theorem \ref{t:ruthsandwr}, one may show that the diagram
\[\xymatrix{\ruth(G)\ar[rr]^{\Phi(G)} \ar[dr]_{\Psi} & & \WR(G) \ar[dl]^{\act} \\ & \VB(G) &}\] 
commutes up to a natural isomorphism.
\end{rem}

\end{document}